\documentclass[11pt,reqno]{amsart}

\usepackage{colordvi}
\usepackage{amsmath,amsfonts,amsthm,amssymb,amsxtra}
\usepackage{amsxtra, amssymb, mathrsfs, color}
\usepackage{graphicx}
\usepackage{colordvi}

\newtheorem{theorem}{Theorem}[section]
\newtheorem{proposition}[theorem]{Proposition}
\newtheorem{conjecture}[theorem]{Conjecture}
\newtheorem{lemma}[theorem]{Lemma}
\newtheorem{corollary}[theorem]{Corollary}

\theoremstyle{definition}

\newtheorem{assumption}[theorem]{Assumption}

\newtheorem{remark}[theorem]{Remark}

\numberwithin{equation}{section}

\definecolor{dgreen}{rgb}{0,0.5,0.2}

\newcommand{\bx}{{\bf x}}
\newcommand{\by}{{\bf y}}
\newcommand{\bz}{{\bf z}}

\newcommand{\U}{\mathcal{U}}
\newcommand{\V}{\mathcal{V}}

\newcommand{\pd}{\partial}
\newcommand{\eps}{\varepsilon}

\newcommand{\R}{\mathbb{R}}

\def\ga{\alpha}     \def\gb{\beta}       
             
                         \def\vge{\varepsilon}
           \def\gh{\eta}
            \def\gl{\lambda}

      \def\gw{\omega}
\def\gx{\xi}

\def\Gw{\Omega}              

\begin{document}


\title[Heat kernels in twisted tubes]{Sharp two-sided heat kernel estimates of twisted tubes and applications}

\author {Gabriele Grillo}

\address {Gabriele Grillo, Dipartimento di Matematica, Politecnico di Milano\\
              Piazza Leonardo da Vinci 32, 20133 Milano, Italy}

\email {gabriele.grillo@polimi.it}

\author {Hynek Kova\v{r}\'{\i}k}

\address {Hynek Kova\v{r}\'{\i}k, DICATAM, Sezione di Matematica, Universit\`a degli studi di Brescia\\
              Via Branze, 38, 25123 Brescia, Italy}

\email {hynek.kovarik@polito.it}

\author{Yehuda Pinchover}

\address{Yehuda Pinchover,
Department of Mathematics, Technion - Israel Institute of Technology\\
           32000 Haifa, Israel}

\email{pincho@techunix.technion.ac.il}

\thanks{G.G. has been partially supported by the MIUR-PRIN 2009 grant ``Metodi di viscosit\`a, geometrici e di controllo per modelli diffusivi nonlineari''. H.K. has been partially supported by the MIUR-PRIN 2010-11 grant for the project  ``Calcolo delle Variazioni''. Y.P. acknowledges the support of the Israel Science
Foundation (grants No. 963/11) founded by the Israel Academy of Sciences and Humanities.}



\begin {abstract}
We prove  on-diagonal bounds for    the heat kernel of the Dirichlet
Laplacian $-\Delta^D_\Omega$ in locally twisted three-dimensional
tubes $\Omega$. In particular, we show that for any fixed $x$ the heat kernel
decays for large times as $\mathrm{e}^{-E_1t}\, t^{-3/2}$, where $E_1$ is the
fundamental eigenvalue of the Dirichlet Laplacian on the cross
section of the tube. This shows that any, suitably regular, local twisting speeds up the decay of the
heat kernel with respect to the case of straight (untwisted) tubes. Moreover, the above large time decay is valid for a wide class of subcritical operators defined on a straight tube.
We also discuss
some applications of this result, such as Sobolev
inequalities and spectral estimates for
Schr\"odinger operators $-\Delta^D_\Omega-V$.
\end{abstract}

\maketitle


\section{Introduction}

\noindent Let $\omega\subset\R^2$ be an open bounded set and let
$\Omega_0 = \omega\times\R$  be a straight tube in $\R^3$. By
separation of variables it is easy to see that the heat kernel of
the Dirichlet Laplacian $-\Delta^D_{\Omega_0}$ on $\Omega_0$
satisfies
\begin{equation} \label{straight-decay}
   k(t,x,y):=    \mathrm{e}^{t E_1}\, \mathrm{e}^{t \Delta^D_{\Omega_0}}(x,y)\  \sim \  t^{-\frac
12}\, \qquad \text{as \ } t\to\infty,
\end{equation}
where $E_1$ is the principal eigenvalue of $-\Delta^D_\omega$, the
Dirichlet Laplacian on $\omega$. Let us now define the twisted tube
$\Omega$ by
$$
\Omega = \{ r_\theta(x_3)\, x\mid  x=(x_1,x_2,x_3) \in \Omega_0\},
$$
where
$$
r_\theta(x_3) = \left( \begin{array}{rcc}
\cos \theta(x_3) & \sin \theta(x_3) & 0 \\
-\sin \theta(x_3) & \cos \theta(x_3) & 0 \\
0 & 0& 1 \end{array} \right)
$$
and $\theta:\R\to\R$ is the angle of rotation. Here and in the
sequel we will denote by $x$ the variable in the straight tube
$\Omega_0$ and by $\bx$ the variable in the twisted tube $\Omega$.
We assume that the support of $\dot\theta$, the derivative  of
$\theta$, is compact, see Section~\ref{sec-prelim} for more details.
It then follows that the spectrum of $-\Delta^D_\Omega$  coincides with  the
half-line $[E_1,\infty)$. Therefore, it is convenient to work with
the shifted operator $-\Delta^D_\Omega-E_1$. This is a nonnegative
self-adjoint operator which generates a contraction, positivity
preserving semigroup $\mathrm{e}^{t(\Delta^D_\Omega+E_1)}$ on $L^2(\Omega)$.
The main object of our interest is its integral kernel
\begin{equation} \label{heat-kernel}
k(t, \bx, \by) := \mathrm{e}^{t (\Delta^D_\Omega+E_1)}(\bx,\by), \quad \bx,\,
\by \in \Omega.
\end{equation}
In particular, we are interested in the influence of twisting on the
long time behavior of $k(t, \bx, \by)$. This is motivated by the
fact that, under appropriate assumptions on $\omega$ and
$\dot\theta$, the Dirichlet Laplacian in the twisted tube $\Omega$
satisfies a Hardy-type inequality
\begin{equation} \label{hardy-operator}
-\Delta^D_\Omega-E_1 \, \geq \, \frac{c}{1+\bx_3^2}
\end{equation}
in the sense of quadratic forms, see \cite{ekk}. One of the
consequences of this inequality is the existence of a finite    positive (minimal)    Green
function of $-\Delta^D_\Omega-E_1$, see e.g.   \cite{grig05,pt}.   Using a different terminology,
the associated semigroup is transient.    On the other hand, \eqref{straight-decay} implies that the associated semigroup of $-\Delta^D_{\Omega_0}-E_1$ corresponding to  the straight tube
is recurrent (see also \cite{f}).    In other
words, inequality \eqref{hardy-operator} implies that $\textstyle\int_0^\infty\, k(t, \bx, \by)\, \mathrm{d}t < \infty$ for all $
\bx\neq \by$,    while the recurrency of $-\Delta^D_{\Omega_0}-E_1$ means that $\textstyle\int_0^\infty\, k(t, x, y)\, \mathrm{d}t = \infty$.  Moreover, since $k(t, \bx, \by)$ and $k(t, \bx, \bx)$ are pointwise equivalent for all $t\geq 1$    \cite[Theorem~10]{da97},   and since $k(t, \bx, \bx)$ is nonincreasing in $t$ it follows that
\begin{equation} \label{green}
k(t, \bx, \bx) = o(t^{-1}) \qquad \text{as \ } t\to\infty.
\end{equation}
This means that the heat kernel of the twisted tube  must
decay faster to zero than the heat kernel of the straight tube given by
\eqref{straight-decay}.  In fact, the validity of the Hardy inequality \eqref{hardy-operator} can be viewed as a qualitative description of the improved decay of $k(t, \bx, \by)$, so that the present investigation of obtaining sharp quantitative bounds for $k(t, \bx, \by)$ is a natural continuation of \cite{ekk}.

Twisted tubes or more generally curved waveguides provide thoughtful models to study various theoretical and experimental physical phenomena, like heat transfer \cite{B,MB}, wave propagation \cite{H}, optics and electromagnetic theory \cite{KO,PR}, and fluid and quantum mechanics \cite{ekk,Po,DJ}. Surprisingly, the connection between twisting and the heat equation was pointed out for the first time only in the recent paper of Krej\v ci\v
r\'{\i}k and Zuazua  \cite{kz}. Let $L^2(\Omega,K)$ be the weighted $L^2$ space with the Gaussian weight $K(x)= \mathrm{e}^{\bx_3^2/4}$. The authors of \cite{kz} proved that for any $a < 3/4$ there exists a constant $C_a$ such that
\begin{equation} \label{kz}
\| \mathrm{e}^{t\, (\Delta^D_\Omega+E_1)}\|_{L^2(\Omega,K)\to L^2(\Omega)} \,
\leq \, C_a\, (1+t)^{-a}, \qquad a < 3/4, \quad \forall \ t\geq 0.
\end{equation}
Notice that $a$ cannot be larger than $1/4$ in a straight tube. Similar result was obtained by the same authors in \cite{kz2} for the heat semigroup of the twisted Dirichlet-Neumann waveguide. From the applicative point of view, we mention that it is known that twisting enhances heat transfer, see for example \cite{B}, \cite{MB}. This pheno\-me\-non seems to be utilized in the so called Twisted Tube technology.

\smallskip

The aim of this paper is to establish sharp pointwise   on-diagonal   heat kernel estimates. In fact,
as one of our main results we will show that
\begin{equation} \label{main-intro}
k(t,\bx,\bx)  \ \asymp  \ \frac{\mbox{dist}(\bx,\partial\Omega)^2}{\sqrt t}\, \,
\min\Big\{\frac{1+\bx_3^2}{t}, \, \, 1\Big\}, \qquad \forall\ t \geq 1,
\end{equation}
see Theorem \ref{main-result}.
Such a two-sided pointwise bound on the heat kernel of course
gives   us   a more detailed information than an integral bound.
Moreover, a simple application of \eqref{main-intro}  allows us to extend inequality \eqref{kz}
to the critical value $a=3/4$ and at the same time to  a  wider class of subspaces of $L^2(\Omega)$
and $L^1(\Omega)$, see Proposition \ref{prop-l2} and Theorem \ref{thm-lp}.

The proof of estimate \eqref{main-intro} relies on the study of positive global solutions of the equation
$(-\Delta^D_\Omega-E_1)u=0$ in $\Gw$   and of suitable functional inequalities on the corresponding
weighted   $L^2$   spaces. We would like to point out that the Hardy inequality \eqref{hardy-operator}
is used only implicitly, to ensure the subcriticality of $-\Delta^D_\Omega-E_1$.
Since heat kernel estimates can
be reformulated in probabilistic terms in term of the survival
probability of the Brownian bridge killed upon exiting $\Omega$,
our results also imply sharp bounds for such probability:
roughly speaking, the $t^{-3/2}$ decay of the kernel for a fixed
$\bx$ might be expressed by saying that the longitudinal part of the
Brownian bridge sees, asymptotically as $t\to\infty$,
the twist as if it were a Dirichlet boundary condition imposed on the cross-section
of the tube. Hence the corresponding heat kernel resembles the one generated by the
Dirichlet Laplacian on a half-line.

As applications of our heat kernel estimates we prove a family of Sobolev-type inequalities for
the operator $-\Delta_\Omega^D-E_1$, see Theorem \ref{thm-sobolev},
and an upper bound on the number of eigenvalues of Schr\"odinger operators
$-\Delta_\Omega^D-V$, where $V$ is an additional electric potential, Theorem \ref{CLR}.
Both these results fail in straight tubes.

While the behavior of the Dirichlet heat kernel on bounded Euclidean regions is well understood (see e.g. \cite{da87}, \cite{zhang} and references therein), much less is known in unbounded regions because
of the great variety of possible geometrical situations. In fact, a rather complete study is available, as far as we know, only in exterior domains, namely in domains of the form $A^c$, $A$ being a compact set with nonempty interior: see \cite{gs2} (and \cite{cms} for some particular cases) for its behavior both in the transient and in the recurrent case when in addition the spatial variables are required to be not too close to the boundary, and \cite{zhang2} for the remaining range, at least in the transient case. See also \cite{DB} and references quoted therein for the study of heat kernel behavior in other special classes of unbounded domains. Note also that the behavior of the heat kernel for $t\le1$ in the class of domains considered in this paper is entirely known from \cite{zhang}.
We would like to mention that the fact that  heat kernels of subcritical operators decay faster than heat kernels of suitably related critical operators has been proved in larger generality in \cite{fkp}, but the general
situation studied there does not allow for quantitative statements.

Our results are not restricted only to twisted tubes. Indeed, if $L: =-\nabla\cdot(a \nabla)+V$ is a uniformly elliptic operator with smooth enough coefficients which is defined on  $\Omega_0$ such that  $L=-\Delta-E_1$ in $\{(x',x_3)\in \Gw_0 \mid |x_3|>R \}$, for some $R>0$ \it and $L$ is subcritical \rm in $\Omega_0$, then
 a straightforward application of our technique yields
\begin{equation}\label{eq-gen}
\exp(-t L)(x,x)  \ \asymp  \ \frac{\mbox{dist}(x,\partial\Omega_0)^2}{\sqrt t}\, \,
\min\Big\{\frac{1+x_3^2}{t}, \, \, 1\Big\}, \qquad \forall\ t \geq 1, \quad \forall\ x\in \Omega_0.
\end{equation} See Subsection~\ref{sec-general} and in particular Theorem \ref{thm-generalization} for a more detailed discussion.

Let us briefly outline the content of the paper. In Section~\ref{sec-prelim} we formulate our main assumptions on $\omega$ and
$\theta$ and fix some necessary notation. The crucial heat kernel
upper bound is proven in Section~\ref{sec-upperb}, see Theorem~\ref{main-ub}. The central idea of the proof is to establish suitable generalized
Nash inequalities on carefully chosen  weighted
$L^2$ spaces and to use the
equivalence between such inequalities and ultracontractivity estimates, cf.  \cite{cou}. Off-diagonal upper bound are then a straightforward consequence of \cite{grig97}.
In Section~\ref{lowerb-sect} we prove the lower bound in \eqref{main-intro}  by means of a
Dirichlet bracketing argument.  Improvements of inequality \eqref{kz} for a larger class of data, including optimal $L^1$ and $L^\infty$ versions, are given in Section~\ref{sec-l2}.  In the closing Section~\ref{sec-clr-sobol} we prove spectral estimates for Schr\"odinger
operators on $\Omega$ and a family of  Hardy-Sobolev type   inequalities for functions from $H^1_0(\Omega)$ (cf. \cite{pt09}). The latter are, similarly as the Hardy inequality \eqref{hardy-operator}, yet another example of functional
inequalities induced by twisting; i.e. they fail in the straight tube $\Omega_0$. The appendices contain some instrumental technical material, mainly on one-dimensional Schr\"odinger operators, and some remarks on Davies' conjecture (see \cite{da97}) in the present case.


\section{\bf Preliminaries}
\label{sec-prelim}

\noindent Throughout the paper we will work under the following
hypotheses on $\omega$ and $\theta$:

\begin{assumption}
$\omega$ is an open bounded connected subset of $\R^2$  with a
$C^2-$regular boundary which contains the origin. Moreover, $\omega$
is not a disc   or a ring centered at    the origin.
\end{assumption}

\begin{assumption}
The function $\theta$ belongs to the class $C^{2,\alpha}(\R)$ with
some $\alpha >0$ and the support of $\dot\theta$ is compact. Without
loss of generality we assume that $\theta(x_3)=0$ for all $x_3 <\,
\inf\, \mathrm{supp}\, \dot\theta$.
\end{assumption}

\noindent Under these assumptions we define the Dirichlet Laplacian
$-\Delta^D_\Omega$ as the unique self-adjoint operator in
$L^2(\Omega)$ generated by the closed quadratic form
\begin{equation}
\int_{\Omega} |\nabla u|^2\, \mathrm{d}\bx  \qquad u\in H^1_0(\Omega).
\end{equation}

As for the notation, given a set $M$ and functions $f_1,f_2:M\to \R_+$
we will use the convention
$$
f_1(z)   \asymp   f_2(z) \quad \Longleftrightarrow \quad \exists \, c>0\, : \, \forall\, z\in M\quad c^{-1} f_1(z) \,
\leq \, f_2(z) \, \leq c \, f_1(z).
$$
Moreover, given a measure $d\mu(x)=\mu(x)\, {\rm d}x$ on $\Omega_0$ and
$p\geq 1$, we denote by $L^p(\Omega_0, \mu)$ the corresponding $L^p$
space with respect to $d\mu$.  The same notation will be used for the
Sobolev spaces $H^1$ and $H^1_0$.
A point $x\in\Gw_0$ will be denoted by $x=(x',x_3)$, where $x'\in \gw$ and $x_3\in \mathbb{R}$.
Set
$$
\gw_a:=\{(x',x_3)\in \Gw_0 \mid x_3=a \}.
$$
We will also need the functions
\begin{equation}
\gamma(t) := \left\{
\begin{array}{l@{\quad}c}
t^{-5/2} , & 0 < t \leq 1 \, , \\
t^{-3/2} , & 1 < t <\infty,
\end{array}
\right. \qquad \Gamma(t) := \left\{
\begin{array}{l@{\quad}c}
t^{-5/2} , & 0 < t \leq 1 \, , \\
t^{-1/2} , & 1 < t <\infty.
\end{array}
\right.
\end{equation}
By the symbol $c$ we will denote a generic positive constant whose value might change from line to line. Finally, we introduce the distance function
$$
\rho(\bx):= \mbox{dist}(\bx, \partial\Omega), \qquad \bx\in \Omega.
$$
We have the following auxiliary result.

\begin{lemma} \label{dist}
Let $\psi_1$ be the normalized principal eigenfunction of $-\Delta^D_\omega$ associated to $E_1$. Let
$T_\theta: \Omega\to \omega$ be defined by
$$
T_\theta(\bx)=  \left(\cos \theta(\bx_3)\, \bx_1 -
\sin \theta(\bx_3)\, \bx_2,\, \sin \theta(\bx_3)\, \bx_1 + \cos
\theta(\bx_3)\, \bx_2 \right).
$$
Then $\psi_1(T_\theta(\bx))\, \asymp \, \rho(\bx)$.
\end{lemma}

\begin{proof}
Let $\Omega_{\bx} = \{\by\in\Omega : \by_3=\bx_3\}$ and define $\tilde\rho(\bx)=$dist$(\bx,\partial \Omega_{\bx})$. Since the boundary of $\omega$ is $C^2$-smooth, the Hopf boundary point lemma,  cf. \cite[Sect.4.6]{da}, implies that  $\psi_1(T_\theta(\bx))\, \asymp \, \tilde\rho(\bx)$.
On the other hand, from the regularity assumptions on $\theta$ it follows that
$\tilde\rho(\bx) \asymp \, \rho(\bx)$.
\end{proof}


\section{Heat kernel bounds}
\label{sec-bounds}

\noindent The main result of  this section is the following

\begin{theorem} \label{main-result}
There exists a constant $c>0$ such that for any $\bx \in \Omega$ and
any $t\geq 1$ we have
 \begin{equation} \label{two-sided-eq}
c^{-1} \ \frac{\rho^2(\bx)}{\sqrt{t}}\,
\min\Big\{\frac{1+\bx_3^2}{t}, \, 1\Big\}\  \leq \  k(t,\bx,\bx)
\ \leq \   c\ \frac{\rho^2(\bx)}{\sqrt{t}}\,
\min\Big\{\frac{1+\bx_3^2}{t}, \, 1\Big\} \,.
\end{equation}
\end{theorem}

\smallskip

\begin{remark}
Note that while $k(t,\bx,\bx) \asymp t^{-3/2}$ as $t\to\infty$ holds pointwise for any $\bx\in\Omega$,
we have $\sup_{\bx}\, k(t,\bx,\bx) \asymp
t^{-1/2}$ as $t\to\infty$. This is caused by the fact that the faster decay of $k(t,\bx,\bx)$ in time is
balanced by the spacial growth of the factor $1+\bx_3^2$.
Similar discrepancy between the behavior of
$k(t,\bx,\bx)$ and $\sup_{\bx}\, k(t,\bx,\bx)$ has been observed also in \cite{da97} for heat kernels on suitable complete manifolds with a finite number of ends with different asymptotic dimensions. Our case shows a similar behavior in a situation in which the two ends of the manifold considered are diffeomorphic (in fact, related by a rigid motion).

\end{remark}

\begin{remark} \label{rem-appl}
It was pointed out in \cite{ekk} that a local twisting has consequences for the transport in quantum waveguides in the
sense that it decreases the probability of a particle being trapped. Theorem \ref{main-result} shows that it has consequence also for the heat transport. Indeed, equation \eqref{two-sided-eq} implies that with a given initial data (satisfying a natural growth condition), the solution to the heat equation in twisted tubes converges faster to the equilibrium state (zero) with respect to the corresponding solution in straight tubes. Hence twisting improves the heat transport, which is in agreement with the numerical results obtained in \cite{B}, \cite{MB}. The presence of the spacial factor $1+\bf x_3^2 $ then shows that the effect of a local twisting decreases  with the distance to the twisted part of the tube.
\end{remark}

\begin{remark}
The behavior of $k(t,\bx,\by)$ for small times is known and, as
expected, is independent of twisting. The following two-sided
estimate is due to \cite{zhang}:
\begin{equation} \label{zhang}
\forall\, t \leq 1:\qquad c^{-1}\, \min\Big\{ \frac{\rho^2(\bx)}{t}\, ,\, \,
1\Big\}\  t^{-\frac 32} \, \leq \,  k(t,\bx,\bx) \, \leq \, c\,
\min\Big\{ \frac{\rho^2(\bx)}{t}\, ,\, \, 1\Big\}\
t^{-\frac 32},
\end{equation}
see also \cite{da87}. This is in accordance with the well-known \it principle of not feeling the boundary \rm of M. Kac (see \cite{K}, \cite{c}, \cite{v}) which can be  loosely stated as:  The Brownian loop joining $x$ to itself in time $t$ does not touch the boundary with a probability tending to one as $t\downarrow0$. Hence our Dirichlet heat kernel is close to the free, $3$-dimensional one for short time.
\end{remark}

\noindent Theorem \ref{main-result} will be proven in several steps in the    following two    subsections.


\subsection{Heat kernel upper bounds}
\label{sec-upperb}

\noindent We shall prove here our main upper bounds for the heat kernel. Since the proof is a bit subtle, let us first outline the proof. The key steps of the proof are on the one hand, Theorem \ref{greenf-equiv} which proves the equivalence of the Green functions for general subcritical operators that coincide far away the twisting region, and on the other hand, suitable weighted Nash-type inequalities (see Lemmas \ref{nash-gen} and \ref{nash-gen2}), which are known to be related to ultracontractive estimates for the heat kernel. More precisely, the equivalence of such Green functions implies the equivalence of the cone of the positive solutions of the shifted  Laplacian that vanish on the boundary of the twisted tube and the corresponding cone of a reference skew-product equation for which such positive solutions are given explicitly. These positive solutions are used as the weights in the corresponding Nash-type inequalities. In view of \cite{cou}, these functional inequalities are equivalent to  our diagonal heat kernel upper bounds.

We introduce the transformation
$$
(U_\theta \, \varphi)(x) := \varphi \left( r_\theta(x_3)\, x \right), \quad x \in \Omega_0, \quad
\varphi \in L^2( \Omega),
$$
which maps $L^2( \Omega)$ unitarily onto $L^2(\Omega_0)$. A
straightforward calculation shows that $H = U_\theta
(-\Delta^D_\Omega)\, U_\theta^{-1}$ is the self-adjoint operator in
$L^2(\Omega_0)$ which acts on its domain as
\begin{equation}
H := -\Delta^D_\omega -(\partial_3+\dot\theta(x_3)\,
\partial_\tau)^2,
\end{equation}
where $\pd_\tau := x_1\pd_2-x_2\pd_1$. The shifted Laplacian
$-\Delta^D_\Omega-E_1$ transforms accordingly into the operator
$$
H_\theta: = H -E_1 = U_\theta (-\Delta^D_{\Omega}-E_1)\,
U_\theta^{-1}, \quad \text{in\, \,} L^2(\Omega_0)
$$
which is generated by the quadratic form
\begin{equation}
\mathcal{Q}[u] := \int_{\Omega_0} \left(|\nabla_\intercal u|^2+|\pd_3 u+
\dot\theta \pd_\tau u|^2- E_1 |u|^2\right){\rm d}x, \quad u\in
D(\mathcal{Q}) = H^1_0(\Omega_0),
\end{equation}  where $\nabla_\intercal:=(\pd_1,\pd_2)$.

We will also consider the reference operator
\begin{equation} \label{model}
A := -\Delta^D_{\Omega_0} +\dot\theta^2(x_3) -E_1 \qquad \text{in\, \, }
L^2(\Omega_0)
\end{equation}
with Dirichlet boundary conditions at $\partial\Omega_0$. Recall the   Hardy-type   inequality
\begin{align} \label{hardy}
\int_{\Omega_0} \left(|\nabla_\intercal u|^2+|\pd_3 u+ \dot\theta \pd_\tau
u|^2- E_1 |u|^2\right){\rm d}x & \geq c_h \int_{\Omega_0} \dot\theta^2\, |u|^2
\, {\rm d}x, \quad \forall\, u\in H^1_0(\Omega_0),
\end{align}
where the constant $c_h>0$ depends on $\dot\theta$ and $\omega$ but
not on $u$, see \cite{ekk}. In the language of criticality theory
this inequality says that $H_\theta$ is a subcritical operator  in $\Omega_0$ (see for example \cite{pt}).   On the other hand, since $-\Delta^D_{\Omega_0} \geq E_1$, it follows from the definition of the operator $A$ that
$$
A \, \geq\,  \dot\theta^2.
$$
Hence $A$ itself is a subcritical operator in $\Omega_0$. We denote the minimal positive Green functions of $H_\theta$ and $A$ in $\Omega_0$ by $G_\theta(x,y)$ and $G_A(x,y)$, respectively.   The following  theorem  plays a crucial role in the proof of our heat kernel upper bounds.

\begin{theorem} \label{greenf-equiv}
Let  $H_1$ and $H_2$ be two subcritical operators in the tube $\Omega_0$ such that  $H_1=H_2$ in $\{(x',x_3)\in \Gw_0 \mid |x_3|>R \}$ for some $R>0$. Let $G_k(x,y)$ be the positive minimal Green function of  $H_k$ in  $\Gw_0$, $k=1,2$.
Assume that the coefficients of $H_1$ and $H_2$ are H\"older continuous in $\{(x',x_3)\in \overline{\Gw_0} \mid |x_3|<R+6 \}$. Then
\begin{equation}\label{eq:green}
G_1\asymp G_2 \qquad \mbox{in } \Gw_0\times \Gw_0 \setminus \{(x,x)\mid x\in \Gw_0\}.
\end{equation}
In particular, there exists a positive constant $C$ such that
\begin{equation}
 C^{-1}\, G_\theta(x,y) \, \leq \, G_A(x,y) \, \leq \, C\, G_\theta(x,y)
\end{equation}
for all $x,y \in \Omega_0$.
\end{theorem}

\begin{proof}
Without loss of generality, we may assume that $H_1=H_2$ in $\{(x',x_3)\in \Gw_0 \mid |x_3|>1 \}$.
By the interior Harnack inequality for $H_k^*$, the formal adjoint of $H_k$ and the behavior of  the Green functions near the singular point we have that
$ G_1(0,(0,0,\pm 2))\asymp G_2(0,(0,0,\pm 2)) $.
Hence, the Harnack boundary principle for $H_1^*=H_2^*$   \cite{a78,CFMS}
implies that
\begin{equation}\label{eq1}
G_1(0,y)\asymp G_2(0,y) \qquad \forall y\in \gw_{\pm 2}.
\end{equation}
Since $G_k(0,y)$ has minimal growth at infinity of $\{y=(y',\gh)\in \partial \Gw_0 \mid \gh >2 \}$ and $\{y=(y',\gh)\in \partial\Gw_0 \mid \gh<-2 \}$ it follows from \eqref{eq1}  \begin{equation}\label{eq2}
G_1(0,y)\asymp G_2(0,y) \qquad \forall y\in \gw_{\gh}, \;|\gh|> 2.
\end{equation}

Now, fix $y\in \gw_{\gh}, \;|\gh|> 3$. Without loss of generality, we may assume $\gh>3$. Then by the Harnack boundary principle for $H_1=H_2$ we have
\begin{equation}\label{eq0}
 \frac{G_1(x,y)}{G_1((0,0,\pm 2),y)}\asymp \frac{G_2(x,y)}{G_2((0,0,\pm 2),y)} \qquad \forall  x\in \gw_{\pm2}\,, \forall y\in \gw_{\gh}, \;|\gh|> 3 .
\end{equation}

Recall that by  the interior Harnack inequality for $H_k$ we have $$G_k((0,0,\pm 2),y)\asymp G_k(0,y).$$ Hence, it follows from \eqref{eq0} that
\begin{equation}\label{eq4}
 \frac{G_1(x,y)}{G_1(0,y)}\asymp \frac{G_2(x,y)}{G_2(0,y)} \qquad \forall  x\in \gw_{\pm2}\,, \forall y\in \gw_{\gh}, \;|\gh|> 3 .
\end{equation}
Combining \eqref{eq2} and \eqref{eq4} we obtain
\begin{equation}\label{eq3}
 G_1(x,y)\asymp G_2(x,y) \quad \forall x\in \gw_{\pm2}, \quad  \forall y\in \gw_{\gh}, \;|\gh|> 3   .
\end{equation}
The minimality of $G_k(\cdot,y)$ in $\{(x',\gx)\in \partial \Gw_0 \mid \gx<-2 \}$ and \eqref{eq3} imply
$$
G_1(x,y)\asymp G_2(x,y) \quad \forall x\in \gw_{\gx}, \;\gx<-2  .
$$
On the other hand, since $G_1(x,y)\asymp G_2(x,y)$ in a small punctured neighborhood of $y$ (the size of the neighborhood depends on $\mathrm{dist}(y,\partial \Gw_0)$), and in light of \eqref{eq3} and the minimality of $G_k$, we obtain  that
$$ G_1(x,y)\asymp G_2(x,y) \quad \forall x\in \gw_{\gx}, \;\gx>2  .$$
So, we obtained
\begin{equation}\label{eq5}
 G_1(x,y)\asymp G_2(x,y) \quad \forall x\in \gw_{\gx}, \;|\gx|\geq 2  \mbox{ and } \forall y\in \gw_{\gh}, \;|\gh|\geq 3 .
\end{equation}
Denote by $G_k^{\Gw_N}(x,y)$ the positive minimal Green function of $H_k$ in $\Gw_N:=\gw\times (-N,N)$, $k=1,2$, $N\leq 6$.
It is known (see for example \cite{a97,hs}) that for a fixed $N$ we have
\begin{equation}\label{eq9}
G_1^{\Gw_N}\asymp G_2^{\Gw_N} \qquad \mbox{in } \Gw_N\times \Gw_N \setminus \{(x,x)\mid x\in \Gw_N\}.
\end{equation}
Fix $N=5$. It follows from the boundary Harnack principle (in $x$)  that for $k=1,2$ we have
\begin{equation}\label{eq10}
 \frac{G_k(x,y)}{G_k((0,0,\gx),y)}\asymp \frac{G_k^{\Gw_5}(x,y)}{G_k^{\Gw_5}((0,0,\gx),y)} \qquad \forall  x\in \gw_{\gx}, \gx=\pm 4\,, \forall y\in \gw_{\gh} , \;|\gh| \leq 3.
\end{equation}
On the other hand,
\begin{equation*}
 G_k((0,0,\gx),0)\asymp 1, \quad G_k^{\Gw_5}((0,0,\gx),0)\asymp 1 \qquad    \gx=\pm 4,\; k=1,2.
\end{equation*}
Hence, the boundary Harnack principle (in $y$) implies that
\begin{equation}\label{eq11}
 G_k((0,0,\gx),y)\asymp G_k^{\Gw_5}((0,0,\gx),y) \qquad \forall  \gx=\pm 4\,, \forall y\in \gw_{\gh} , \;|\gh|\leq 3.
\end{equation}
Consequently, \eqref{eq10}, \eqref{eq11}, the behavior of Green functions near the singularity, and the comparison principle imply that
\begin{equation}\label{eq12}
 G_k(x,y)\asymp G_k^{\Gw_5}(x,y) \qquad \forall  x\in \gw_{\gx}, |\gx|\leq 4\,, \forall y\in \gw_{\gh} , \;|\gh|\leq 3.
\end{equation}
In light of \eqref{eq12} and \eqref{eq9} with $N=5$ we obtain
\begin{equation}\label{eq14}
 G_1(x,y)\asymp G_2(x,y) \quad \forall x\in \gw_{\gx}, \;|\gx|\leq 4,    \mbox{ and } \forall y\in \gw_{\gh}, \;|\gh|\leq 3 .
\end{equation}
Since $G_k$ has minimal growth at infinity, it follows from \eqref{eq14} that
\begin{equation}\label{eq15}\begin{split}
 &G_1(x,y)\!\asymp \!G_2(x,y)\ \ {\rm if} \\ &(x,y)\!\in\! \{ x\in \gw_{\gx}, |\gx|\geq 4;  y\in \gw_{\gh}, |\gh|\leq 3\}\cup \{ x\in \gw_{\gx}, |\gx|\leq 2;  y\in \gw_{\gh}, |\gh|\geq 3\}.\end{split}
\end{equation}
Thus, \eqref{eq5}, \eqref{eq14}, and \eqref{eq15} imply \eqref{eq:green}.
\end{proof}

\begin{remark}\label{rem:Murata}
Let $M$ be a noncompact smooth Riemannian manifold and let $\Gw$ and $\Gw_j$, $j=0,1,\ldots,\ell$ be subdomains of $M$ with Lipschitz boundaries such that
$$ \Gw=\bigcup_{j=0}^\ell\Gw_j, \quad \Gw_i\cap \Gw_j=\emptyset \quad \forall\, 1\leq i<j \leq \ell,$$
and such that $\overline{\Gw_0}$
 is compact in $M$. Let $H_1$ and $H_2$ be two subcritical operators in $\Omega$ such that  $H_1=H_2$ in $\bigcup_{j=1}^\ell\Gw_j$, and let $G_k(x,y)$ be the positive minimal Green function of  $H_k$ in  $\Gw$, $k=1,2$ (cf. \cite[Section~7]{m}).
Assume that the coefficients of $H_1$ and $H_2$ are H\"older continuous in $\overline{\Gw_0}$. By adopting the proof of Theorem~\ref{greenf-equiv} we obtain that
$$G_1\asymp G_2 \qquad \mbox{in } \Gw\times \Gw \setminus \{(x,x)\mid x\in \Gw\}.$$
\end{remark}

The subcriticality of the operator $-\frac{d^2}{dr^2}\, + \dot\theta^2(r)$ on $\R$ implies that there are exactly two positive minimal solutions (in the sense of Martin boundary) $g_j$, $j=1,2$ of the equation $(-\frac{d^2}{dr^2}\, + \dot\theta^2(r))g=0$ in $\mathbb{R}$ satisfying $g_j(0)=1$
\cite[Appendix 1]{m86}. Moreover, we may assume that
\begin{equation} \label{gj-asymp}
\quad g_1(x_3) \, \asymp\, 1 + \Theta(-x_3)\, |x_3|, \quad g_2(x_3) \, \asymp\, 1+  \Theta(x_3)\, |x_3|,
\end{equation}
where $\Theta$ is the Heaviside function. Let
\begin{equation}\label{cauchy}
g_0:=(g_1+g_2)/2\,.
\end{equation}
Clearly, we have
\begin{equation} \label{g0-asymp}
g_0(x_3) \, \asymp\, 1+|x_3|.
\end{equation}
The functions $w_j:\Omega\to\R$ given by
\begin{equation} \label{w-def}
w_j(x) := \psi_1(x_1,x_2)\,  g_j(x_3),    \quad j=0,1,2,
\end{equation}
then satisfy
\begin{equation} \label{w-eq}
A\,  w_j = 0, \quad w_j >0 \quad \text{in\, } \Omega_0, \quad w_j = 0 \quad \text{on \ } \partial\Omega_0.
\end{equation}
   We note that for any positive solution $w$ of the equation $A\,  w = 0$ on $\Gw_0$ that vanishes on $\partial \Gw_0$ there exists a unique pair of nonnegative numbers $\ga$ and $\gb$ such that $w=\ga w_1+\gb w_2$ \cite[Theorem~7.1]{m}.

\smallskip

\noindent
Next, we apply the above crucial results to obtain the following lemma.

\begin{lemma} \label{gr-state}
There exist positive functions $v_j\in C^2(\Omega_0), j=0,1,2,$ such that
\begin{equation} \label{gr-eq}
H_\theta\, v_j = 0, \qquad  v_j(x) \, \asymp \, w_j(x).
\end{equation}
\end{lemma}

\begin{proof}
This follows from Theorem~\ref{greenf-equiv} and
\cite[Lemma~2.4]{p88}.
\end{proof}

\noindent With this result at hand, we define the (ground state) transformation
\begin{equation}
\U_0 : L^2(\Omega_0)\to L^2(\Omega_0, v_0^2),\qquad  (\U_0\, u)(x) := v_0^{-1}(x)\, u(x), \quad x\in\Omega_0.
\end{equation}
$\U_0$ maps $L^2(\Omega_0)$ unitarily onto $L^2(\Omega_0,v_0^2)$ and
$\mathcal{Q}[u]$ transforms into the closed quadratic form
\begin{equation}
\begin{aligned}
Q_0[f] := \mathcal{Q}[v_0 f]= \int_{\Omega_0} \left(|\nabla_\intercal
f|^2+|\pd_3 f+ \dot\theta\, \pd_\tau f|^2\right)\, v_0^2\,  \mathrm{d}x,\\  f\in
D(Q_0)= H^1(\Omega_0, w_0^2).
\end{aligned}
\end{equation}
 The fact that the form domain $D(Q_0)$
coincides with $H^1(\Omega_0, w_0^2)$ follows from the regularity of
$\omega$, see \cite{ds}, and from the equivalence
$$
c\, |\nabla f|^2 \, \leq \, |\nabla_\intercal f|^2+|\pd_3 f+ \dot\theta\,
\pd_\tau f|^2 \, \leq \, c^{-1} \, |\nabla f|^2.
$$
The upper bound is immediate. The lower bound will be given in the
proof of Proposition \ref{hk-twisted}.

We denote by $B_0$ the self-adjoint operator in $L^2(\Omega_0, v_0^2)$
associated with the form $Q_0[f]$. By standard arguments, see e.g.
\cite[Section~4.7]{da}, it follows that the semigroup $\exp(-t B_0)$ is symmetric and
submarkovian on $L^2(\Omega_0, v_0^2)$ and since
$$
H_\theta= \U_0^{-1}B_0\, \U_0,
$$
we get
\begin{equation} \label{transf-heat}
\mathrm{e}^{-t H_\theta}(x,y) = v_0(x)\, v_0(y)\,    \mathrm{e}^{-t B_0}(x,y).
\end{equation}

\noindent Let $\lambda>1$ and introduce a $C^1$ decreasing bijection
$m_\lambda$ of $\R_+$ onto itself by
\begin{equation} \label{m}
m_\lambda(t): =   \lambda  \left\{
\begin{array}{l@{\quad}c}
 t^{-5/2}, & 0 < t \leq 1/2 \, , \\
  \chi(t), &  1/2 < t \leq 1, \\
 t^{-3/2} , & 1 <  t <\infty,
\end{array}
\right.
\end{equation}
where $\chi$ is a $C^1$ decreasing convex function chosen such that
$m_\lambda(t)$ is $C^1(\R_+)$.  Next we define
\begin{equation} \label{xi}
\xi_\lambda(r): = -m_\lambda'(m_\lambda^{-1}(r)), \quad r\in\R_+.
\end{equation}

\noindent We have

\begin{lemma} \label{nash-gen}
There exists $\lambda_0>0$ such that the inequality
\begin{equation}\label{nash-gen-eq}
 \xi_\lambda\big(\|f\|_{L^2(\Gw_0,w_0^2)}^2\big) \, \leq \,
\int_{\Omega_0} |\nabla f|^2\, w_0^2\, {\rm d}x
\end{equation}
holds for all $f\in H^1(\Omega_0,w_0^2)\cap L^1(\Omega_0,w_0^2)$ with
$\|f\|_{L^1(\Omega_0,w_0^2)}\leq 1$ and all $\lambda > \lambda_0$.
\end{lemma}

\begin{proof}
Consider the heat kernel $\mathrm{e}^{-t A}(x,y)$ of the operator $A$. Since
$$
A = (-\Delta^D_\omega-E_1) \otimes 1 + 1\otimes (-\pd_3^2
+\dot\theta^2),
$$
where $\otimes$ denotes the tensor product.  We have
\begin{equation} \label{hk-decomp}
\exp(-t A)(x,y) = \sum_{j=1}^\infty\, \mathrm{e}^{t(E_1-E_j)}\,
\psi_j(x_1,x_2)\, \psi_j (y_1,y_2)\, q(t,x_3,y_3),
\end{equation}
where $q(t,r,s)$ is the heat kernel of the one-dimensional
Schr\"odinger operator
\begin{equation}
-\frac{d^2}{dr^2}\, + \dot\theta^2(r) \quad \text{in\, } L^2(\R),
\end{equation}
and $E_j$ and $\psi_j$ are the eigenvalues and (normalized)
eigenfunctions of $-\Delta^D_\omega$. By Proposition \ref{prop-app}, see Appendix~\ref{sect-1dim}    (cf. \cite[Theorem~4.2]{m84}),   there
exists a    positive    constant $c$ such that
\begin{equation} \label{q-bound}
q(t,r,r) \leq \frac{c\, g_0^2(r)}{t^{3/2}} \quad \text{if\, } t\geq 1,
\qquad q(t,r,r) \leq \frac{c}{\sqrt{t}} \quad \text{if\, } 0<t< 1,
\end{equation}
where $g_0$ is the function defined by \eqref{cauchy}. On the other
hand, by the ultracontractivity of $\mathrm{e}^{t \Delta^D_\omega}$ we have \cite[Theorem~4.2.5]{da}
\begin{equation} \label{2dim-b1}
\sum_{j=1}^\infty\, \mathrm{e}^{t(E_1-E_j)}\, \psi_j^2(x_1,x_2) \ \leq  \ c\,
\psi_1^2(x_1,x_2) \qquad  t\geq 1.
\end{equation}
\noindent Finally, by \cite[Theorem~4.6.2]{da} we have
\begin{equation} \label{2dim-b2}
\mathrm{e}^{t \Delta^D_\omega}  ((x_1,x_2),(x_1,x_2))   \, \mathrm{e}^{t E_1} =
\sum_{j=1}^\infty\, \mathrm{e}^{t(E_1-E_j)}\, \psi_j^2(x_1,x_2) \leq \frac
{c}{t^2}\, \, \psi_1^2(x_1,x_2) \quad  0<t< 1.
\end{equation}
Combining all these estimates gives
\begin{equation}\label{A-ub}
\mathrm{e}^{-t\, A}(x,x) \leq \, c\, \psi_1^2(x_1,x_2)\, g_0^2(x_3)\, \gamma(t)   =c\, w_0^2(x)\, \gamma(t)
\end{equation}
Next, mimicking the construction of the operator $B_0$ above we notice
that the operator $\widetilde{A}_0:= \V_0 \, A\,  \V_0^{-1}$, where $\V_0$ is the unitary
(ground state) transformation $\V_0 : L^2(\Omega_0)\to L^2(\Omega_0, w_0^2)$ acting
as
\begin{equation}
(\V_0\, u)(x) := w_0^{-1}(x)\, u(x), \quad x\in\Omega_0,
\end{equation}
is associated with the closed quadratic form
\begin{equation} \label{q-tilde}
\widetilde{Q}_0[f]: = \int_{\Omega_0} |\nabla f|^2\, w_0^2\, {\rm d}x, \qquad f \in
D(\widetilde{Q}_0) = H^1(\Omega_0, w_0^2).
\end{equation}
By \eqref{A-ub} we then get
\begin{equation}
\sup_{x}\, \mathrm{e}^{-t\, \widetilde{A}_0}(x,x) = \sup_{x}\, \frac{\mathrm{e}^{-t\,
A}(x,x)}{w_0^2(x)}\, \leq c\, \gamma(t)
\end{equation}
for all $t>0$. Hence,    in view of \eqref{semigroup},    we get
\begin{equation}\label{mlambda}
\|\mathrm{e}^{-t \widetilde A_0} \|_{L^1(\Omega_0, w_0^2) \to L^\infty(\Omega_0,w_0^2)} \ \leq \
m_\lambda(t)
\end{equation}
if $\lambda$ in \eqref{m} is chosen large enough. Note that $(-\log m_\lambda(t))'$ has a polynomial growth.  Therefore, \eqref{mlambda} and \cite[Proposition~II.4]{cou}  yield
\begin{equation} \label{nash-unit}
\xi_\lambda\big(\|f\|_{L^2(\Gw_0,w_0^2)}^2\big) \, \leq \,
\int_{\Omega_0} |\nabla f|^2\, w_0^2\, {\rm d}x \qquad \forall\, f\in
C^\infty_0(\Omega_0)\, : \, \|f\|_{L^1(\Omega_0,w_0^2)} \leq 1.
\end{equation}
Hence \eqref{nash-gen-eq} follows by density.
\end{proof}

\noindent In a similar way as we introduced the functions $m_\lambda$ and $\xi_\lambda$ we define
\begin{equation} \label{mu}
\mu_\lambda(t) := \lambda \left\{
\begin{array}{l@{\quad}c}
 t^{-5/2}, & 0 < t \leq 1/2 \, , \\
 \tilde\chi(t), &  1/2 < t \leq 1, \\
  t^{-1/2} , & 1 <  t <\infty,
\end{array}
\right.
\end{equation}
where $\tilde\chi$ is a $C^1$ decreasing convex function chosen such that
$\mu_\lambda(t)$ is $C^1(\R_+)$. Accordingly, we define
\begin{equation} \label{theta}
\vartheta_\lambda(r) := -\mu_\lambda'(\mu_\lambda^{-1}(r)), \quad r\in\R_+.
\end{equation}

\begin{lemma} \label{nash-gen2}
There exist $\lambda_j>0,\, j=1,2,$ such that the inequality
\begin{equation}\label{nash-gen-eq-2}
\vartheta_\lambda\big(\|f\|_{L^2(\Gw_0,w_j^2)}^2\big) \, \leq \,
\int_{\Omega_0} |\nabla f|^2\, w_j^2\, {\rm d}x , \quad j=1,2
\end{equation}
holds for all $f\in H^1(\Omega_0,w_j^2)\cap L^1(\Omega_0,w_j^2)$ with
$\|f\|_{L^1(\Omega_0,w_j^2)}\leq 1$ and all $\lambda > \lambda_j$.
\end{lemma}

\begin{proof}
We introduce operators
$\widetilde{A}_j := \V_j \, A\, \V_j^{-1}$, where $\V_j,\, j=1,2$ are unitary
transformations $\V_j : L^2(\Omega_0)\to L^2(\Omega_0, w_j^2)$ which act
as
\begin{equation}
(\V_j\, u)(x) := w_j^{-1}(x)\, u(x), \quad x\in\Omega_0.
\end{equation}
These operators are associated with closed quadratic forms
\begin{equation} \label{q-tilde-2}
\widetilde{Q}_j[f]: = \int_{\Omega_0} |\nabla f|^2\, w_j^2\, {\rm d}x, \qquad f \in
D(\widetilde{Q}_j) = H^1(\Omega_0, w_j^2).
\end{equation}
We follow the arguments of the proof of Lemma \ref{nash-gen} replacing \eqref{q-bound} by
$$
q(t,r,r) \, \leq \, \frac{1}{\sqrt{4\pi\, t}}\, \leq \, \frac{c\, g_j^2(r)}{\sqrt{t}} \qquad \forall \ t\geq 0,
$$
which follows from Proposition \ref{prop-app} given in Appendix~\ref{sect-1dim}. This leads to
\begin{equation}\label{eq:aj}
\sup_{x}\, \mathrm{e}^{-t\, \widetilde{A}_j}(x,x) = \sup_{x}\, \frac{\mathrm{e}^{-t\,
A}(x,x)}{w_j^2(x)}\, \leq c\, \Gamma(t), \quad j=1,2
\end{equation}
for all $t>0$, and therefore, if $\lambda$ in \eqref{mu} is chosen large enough, then
$$
\|\mathrm{e}^{-t \widetilde A_j} \|_{L^1(\Omega_0, w_j^2) \to L^\infty(\Omega_0,w_j^2)} \ \leq \
\mu_\lambda(t).
$$
The statement then follows as in the proof of Lemma \ref{nash-gen}.
\end{proof}

\subsubsection{On-diagonal upper bounds}
The functional inequalities proven in the previous Lemmata enable us to prove  the  following on-diagonal heat kernel estimates.

\begin{proposition} \label{hk-twisted}
There exists a constant $C$ such that for any $x\in\Omega_0$ and any
$t>0$ the following inequality holds:
\begin{align} \label{twisted-ub}
\mathrm{e}^{-t H_\theta}(x,x) \leq C\, \psi_1^2(x_1,x_2)\, (1+x_3^2)\,
\gamma(t)\, .
\end{align}
\end{proposition}

\begin{proof}
We note that $|\pd_\tau f|^2 \leq C_\omega\, |\nabla_\intercal f|^2$ for
some constant $C_\omega$. Using the inequality
$$
2 |\dot\theta|\, |\pd_3 f|\, |\pd_\tau f| \leq \eps\, |\pd_3 f|^2
+\eps^{-1} |\dot\theta|^2\, |\pd_\tau f|^2, \quad 0<\eps<1,
$$
and taking $\vge$ close to $1$, it is then easy to see that
\begin{equation} \label{aux-new}
|\nabla_\intercal f|^2+|\pd_3 f+ \dot\theta \pd_\tau f|^2 \geq  c_0 \, |\nabla
f|^2
\end{equation}
for some  $c_0>0$.
 Let $\|f\|_{L^1(\Omega_0,v_0^2)}
\leq 1$. By Lemma \ref{gr-state} we have $\kappa_0^{-1} \leq v_0^2/w_0^2 \leq \kappa_0$ for some $\kappa_0>1$. We apply Lemma \ref{nash-gen} to the function  $\tilde f:=\kappa_0^{-1} f$.  Using the fact that $\xi_\lambda$ is increasing, in view of Lemma \ref{xi-lem}, see Appendix~\ref{sect-xi},  and \eqref{aux-new} we obtain
\begin{align*}
\xi_\lambda\big(\, \|f\|_{  L^2(\Gw_0,v_0^2)  }^2\big) & \leq \xi_\lambda\big(\, \kappa_0^3\,
\|\tilde f\|_{L^2(\Gw_0,w_0^2)}^2\big) \leq C_{\kappa_0^3}\ \xi_\lambda\big(
\|\tilde f\|_{L^2(\Gw_0,w_0^2)}^2\big) \\
& \leq C_{\kappa_0^3}\ \kappa_0^{-2}  \int_{\Omega_0} |\nabla f|^2\, w_0^2\, {\rm d}x \, \leq \, a_0^{-1} \, Q_0[f], \qquad a_0 := c_0\, \kappa_0\, C_{\kappa_0^3}^{-1},
\end{align*}
where $c_0$ is the constant in \eqref{aux-new}.
Hence
\begin{equation} \label{nash-final}
 a_0\,  \, \xi_\lambda\big(\|f\|_{L^2(\Gw_0,v_0^2)}^2\big) \, \leq \, Q_0[f]
\qquad \forall\, f\in C^\infty_0(\Omega_0)\, : \, \|f\|_{L^1(\Omega_0,v_0^2)}
\leq 1.
\end{equation}
This inequality extends by
density to all functions $f\in H^1(\Omega_0, v_0^2)$ satisfying the condition
$\|f\|_{L^1(\Omega_0,v_0^2)} \leq 1$. By the standard Beurling-Deny criteria and the result of \cite[Thm.1.3.3]{da} it follows
that the operator $B_0$ associated with the form $Q_0$ generates  a positivity preserving semigroup $\mathrm{e}^{-t B_0}$ which is contractive in $L^p(\Omega_0, v_0^2)$ for all $p\in [1,\infty]$ and all $t\geq 0$. These facts and the integrability at infinity of $1/\xi_\lambda$ (see Appendix~\ref{sect-xi}) allow us to apply \cite[Proposition~II.1]{cou} which, in view of \eqref{nash-final}, gives
\begin{equation} \label{aux}
\|\mathrm{e}^{-t B_0}\|_{L^1(\Omega_0, v_0^2) \to L^\infty(\Omega_0,v_0^2)}\, \leq
\, m_\lambda(a_0 \, t).
\end{equation}
Equation \eqref{twisted-ub} thus follows by applying \eqref{transf-heat}.
\end{proof}

\begin{remark} As expected, the twisting influences the decay rate of $\mathrm{e}^{-t
H_\theta}(x,x)$ for large times. On the other hand, the faster decay
in time is compensated by the additional weight factor $(1+x_3^2)$.
>From our heat kernel lower bounds, see Theorem \ref{thm-lowerb-1},
it follows that the growth of this weight cannot be
improved.
\end{remark}

\noindent The next result holds for twisted as well as for
straight tubes, i.e.\! for $\dot\theta\equiv 0$.

\begin{proposition} \label{non-twist-ub}
There exists a constant $C$ such that for any $x\in\Omega_0$ and any
$t>0$
\begin{align} \label{non-twist-eq}
\mathrm{e}^{-t H_\theta}(x,x) \leq C\, \psi_1^2(x_1,x_2)\, \Gamma(t).
\end{align}
\end{proposition}

\begin{proof}
We define the (ground state) transformations $\U_j : L^2(\Omega_0)\to L^2(\Omega_0, v_j^2)$ by
\begin{equation}
(\U_j\, u)(x) := v_j^{-1}(x)\, u(x), \quad x\in\Omega_0, \quad j=1,2.
\end{equation}
Hence $\U_j$ map $L^2(\Omega_0)$ unitarily onto $L^2(\Omega_0,v_j^2)$ and
$\mathcal{Q}[u]$ transforms into
\begin{equation}
Q_j[f] := \mathcal{Q}[v_j f]= \int_{\Omega_0} \left(|\nabla_\intercal
f|^2+|\pd_3 f+ \dot\theta\, \pd_\tau f|^2\right)\, v_j^2\,  \mathrm{d}x, \, f\in
D(Q_j)= H^1(\Omega_0, w_j^2),
\end{equation}
Accordingly, we introduce operators $B_j :=  \U_j\, H_\theta\, \U_j^{-1}$ generated by the quadratic forms $Q_j$. As above, we get
\begin{equation} \label{transf-2}
\mathrm{e}^{-t H_\theta}(x,y) = v_j(x)\, v_j(y)\, \mathrm{e}^{-t B_j}(x,y), \quad j=1,2.
\end{equation}
In the same way as in the proof of Proposition \ref{hk-twisted}   (using Lemma~\ref{nash-gen2})   we thus arrive at
$$
\|\mathrm{e}^{-t B_j}\|_{L^1(\Omega_0, v_j^2) \to L^\infty(\Omega_j,v_j^2)}\, \leq
\,  \mu_\lambda( a_j \, t), \quad j=1,2.
$$
where $a_j>0$. Hence by \eqref{transf-2}
$$
\mathrm{e}^{-t H_\theta}(x,x) \, \leq c\, \psi_1^2(x_1,x_2)\, g_1^2(x_3)\, \Gamma(t), \quad \mathrm{e}^{-t H_\theta}(x,x) \, \leq c\, \psi_1^2(x_1,x_2)\, g_2^2(x_3)\, \Gamma(t)
$$
for all $x\in\Omega_0$ and $t>0$.
This concludes the proof.
\end{proof}

\begin{theorem} \label{main-ub}
There exists a constant $C>0$ such that for any $\bx\in\Omega$ and any
$t\geq 1$ the following inequalities hold true
\begin{align}
k(t,\bx,\bx)  & \, \leq  \, C\, \rho^2(\bx)\,
\min\big\{(1+\bx_3^2)  \, t^{-\frac 32}, \, \, t^{-\frac 12}\big\}.
\label{main-ub-eq2}
\end{align}
\end{theorem}

\begin{proof}
Let $\bx, \by \in \Omega$. From
$$
U^{-1}_\theta\, \mathrm{e}^{-t H_\theta} \, U_\theta = \mathrm{e}^{t\,
(\Delta^D_\Omega+E_1)}
$$
we get
$$
k(t,\bx,\by) =  \mathrm{e}^{-t H_\theta} (r_\theta^{-1} \bx ,\, r_\theta^{-1}
\by).
$$
The statement thus follows directly from Propositions \ref{hk-twisted},  \ref{non-twist-ub} and Lemma \ref{dist}.
\end{proof}

\subsubsection{Off-diagonal upper bounds}

\noindent A combination of \eqref{zhang} with Theorem~\ref{main-ub} gives

\begin{corollary} \label{upperb-thm}
For any $C>4$ there exists a constant $K_C>0$ such that for any $\bx\in\Omega$ and any
$t\geq 1$ it holds
\begin{align}
k(t,\bx,\by)  & \, \leq  \, K_C\, \rho(\bx)\, \rho(\by)\,
\min\Big\{\sqrt{(1+\bx_3^2) (1+\by_3^2)} \, \, t^{-\frac 32}, \, \,
t^{-\frac 12}\Big\}\, \, \mathrm{e}^{-\frac{|\bx-\by|^2}{Ct}},
\end{align}
\end{corollary}

\begin{proof}
>From \eqref{zhang} and Theorem~\ref{main-ub} it follows that
$$
k(t,\bx,\bx)  \leq  C\, \rho^2(\bx)\,  (1+\bx_3^2) \, \gamma(t) \qquad \forall\, \bx\in\Omega, \,
\, \forall\, t>0.
$$
A direct inspection shows that  \cite[Theorem~3.1]{grig97} is
applicable to $k(t,\bx, \by)$ with the respective functions $f$ and
$g$ which parametrically depending on $\bx$ and $\by$    (see the example in \cite[p.~37]{grig97}).
Hence for any $C>4$ and all $t>0$ it holds
$$
k(t,\bx,\by)   \leq   \delta(C)\, \rho(\bx) \rho(\by)
\sqrt{(1+\bx_3^2) (1+\by_3^2)} \, \, \gamma(t) \,
\mathrm{e}^{-\frac{r(\bx,\by)^2}{Ct}}.
$$
where $r(\bx,\by)$ is the geodesic distance between $\bx$ and $\by$
and $\delta(C)$ is a positive constant which depends on $C$ and
$\gamma(\cdot)$. Repeating the same procedure with the bound
$$
k(t,\bx,\bx)  \leq  C\, \rho^2(\bx)\,  \Gamma(t) \qquad \forall\, \bx\in\Omega, \, \, \forall\, t>0,
$$
which again follows from \eqref{zhang} and
Theorem~\ref{main-ub},  we obtain
$$
k(t,\bx,\by)  \leq  \tilde\delta(C)\, \rho(\bx) \rho(\by) \,
\Gamma(t) \, \mathrm{e}^{-\frac{r(\bx,\by)^2}{Ct}}.
$$
The fact that $r(\bx,\by) \geq |\bx-\by|$ completes the proof.
\end{proof}
\begin{remark}
By \cite{zhang}, there exist positive constants $c$, $C$ and $T$ such that for any $\bx, \by\in\Omega$ and any
$0<t\leq T$ the following off-diagonal estimates holds true
\begin{align}\label{zhao_off}
\min\Big\{\frac{\rho(\bx)\, \rho(\by)}{t}, 1\Big\}  \frac{c\,\mathrm{e}^{-\frac{C|\bx-\by|^2}{t}}}{t^{3/2}} \leq k(t,\bx,\by)   \leq
\min\Big\{\frac{\rho(\bx)\, \rho(\by)}{t}, 1\Big\} \frac{\mathrm{e}^{-\frac{|\bx-\by|^2}{Ct}}}{ct^{3/2}}\,.
\end{align}
\end{remark}

\vspace{0.2cm}


\subsection{Heat kernel lower bounds and the Brownian bridge reformulation}
\label{lowerb-sect}

The aim of this section is to show that the long time decay rate
$t^{-3/2}$ of the upper bound \eqref{twisted-ub} is sharp. In fact, we show that a matching lower bound holds. The proof of such a lower bound is easier than the previous ones. It uses monotonicity of the Dirichlet heat kernel with respect to the domain to compare the heat kernel $k(t,\bx,\bx)$ of the twisted tube with the Dirichlet heat kernel of a straight half-tube. In fact, we have

\begin{theorem} \label{thm-lowerb-1}
There exists a positive constant $c$ such that for any $\bx\in\Omega$ and any
$t\geq 1$  it holds
\begin{align}
k(t,\bx,\bx)  & \, \geq  \, c\, \rho(\bx)^2 \,
\min\Big\{(1+\bx_3^2) \, t^{-\frac 32}, \, \, t^{-\frac 12}\Big\}
\label{main-ub-eq}.
\end{align}
\end{theorem}

\begin{proof}
We start by proving that for any $\bx\in\Omega$ with $|x_3| >R+1$
and any $t\geq 1$ we have
\begin{equation} \label{first-lowerb}
k(t,\bx,\bx) \geq \, C\, \rho(\bx)^2 \, t^{-\frac 12}\,
\min\Big\{1,\, \frac{\bx_3^2}{t}\, \Big\}, \qquad C>0.
\end{equation}
Suppose that $\bx_3<-(R+1)$. To get a lower bound on $k(t,\bx,\bx)$
we impose additional Dirichlet boundary conditions at    $\omega_{-R}$,    and denote by $\widetilde k(t,\bx,\by)$ the heat kernel of
the    Laplacian  on $\gw \times (-\infty,-R))$.    In view of the reflection principle,
see e.g. \cite[Section~4.1]{da} and the ultracontractivity of $\mathrm{e}^{t \Delta^D_\omega}$, we get
\begin{align} \label{refl}
 k(t,\bx,\by) & \geq \widetilde
k(t,\bx,\by) \nonumber\\
&= \frac{1}{\sqrt{4\pi t}}\, \,
\Big(\mathrm{e}^{-\frac{(\bx_3-\by_3)^2}{4t}}-\mathrm{e}^{-\frac{(\bx_3+\by_3+2R)^2}{4t}}
\Big)\, \sum_{j\geq 1}\, \mathrm{e}^{(E_1-E_j)t}\, \psi_j(\bx_1,\bx_2)
\psi_j(\by_1,\by_2)\nonumber\\ &\geq C\frac{1}{\sqrt{4\pi t}}\, \,
\Big(\mathrm{e}^{-\frac{(\bx_3-\by_3)^2}{4t}}-\mathrm{e}^{-\frac{(\bx_3+\by_3+2R)^2}{4t}}\Big)
\psi_1(\bx_1,\bx_2)\psi_1(\by_1,\by_2)
\end{align}
for all $\by\in\Omega$ with $\by_3< -R-1$. Using the inequality
\begin{equation} \label{aux-eq}
1-\mathrm{e}^{-z}\, \geq (1-\mathrm{e}^{-1})\, \min\{ 1,\, z\}, \quad z\geq 0
\end{equation}
we thus get
\[
k(t,\bx,\bx) \geq   c   \frac{1-\mathrm{e}^{-1}}{\sqrt{4\pi}}\, \,
\psi_1^2(\bx_1,\bx_2)\, \, t^{-\frac 12}\, \min\Big\{1,\,
\frac{(\bx_3+R)^2}{  t  }\, \Big\}.
\]
Taking into account Lemma \ref{dist} and the elementary inequality
$$
(\bx_3+R)^2\geq \left(\frac{\bx_3}{R+1}\right)^2 \qquad \forall\, \bx_3 <-R-1,
$$
we obtain \eqref{first-lowerb} for $\bx_3 < -R-1$.
The proof of the corresponding lower bound for $\bx_3 >R+1$ is
completely analogous.

In order to treat the case $|x_3| \leq R+1$, we fix a $\by_0$ such
that $(\by_0)_3<-(R+1)$, and a number $\varepsilon <
\min\{\rho^2(\by_0), 1\}/4$. We then use the semigroup property to
get, for any $t>1$:
\[
\begin{aligned}
k(t,\bx,\bx)&=\int_{\Omega\times\Omega} k\Big(\frac13,\bx,\by\Big)k
\Big(t-\frac23,\by,\bz\Big)k\Big(\frac13,\bz,\bx\Big)\, {\rm d}\by\,{\rm
d}\bz \\
&\geq\int_{B(\by_0,\varepsilon)\times B(\by_0,\varepsilon)}
k\Big(\frac13,\bx,\by\Big) \widetilde k  \Big(t-\frac23,\by,\bz\Big)k\Big(\frac13,\bz,\bx\Big)\, {\rm d}\by\,{\rm
d}\bz .
\end{aligned}
\]
To bound from below the terms involving the time $s=1/3$, we use
Zhang's off--diagonal lower bound \eqref{zhao_off}. From the choice of $\by_0$ and $\varepsilon$ and
it follows that
$$
k\Big(\frac13,\bx,\by\Big) \geq C_1(\by_0,\varepsilon)\,
\rho(\bx),\quad k\Big(\frac13,\bz,\bx\Big) \geq
C_2(\by_0,\varepsilon)\, \rho(\bx), \quad \forall\, \by,\bz \in
B(\by_0,\varepsilon).
$$
Here we used the fact that $|\bx-\by|^2 +|\bx-\bz|^2$ is bounded
from above since $|x_3| \leq R+1$. Hence
\[
k(t,\bx,\bx) \geq \, C\rho^2(\bx)\int_{B(\by_0,\varepsilon)\times
B(\by_0,\varepsilon)} \widetilde
k\big(t-\frac23,\by,\bz\big)\, {\rm d}\by\,{\rm
d}\bz\
\]
>From \eqref{refl} and \eqref{aux-eq} we get
$$
\widetilde k\big(t-\frac23,\by,\bz\big) \geq C\, t^{-\frac 12}\,
\min\Big\{1,\, \frac{1+\bx_3^2}{t}\, \Big\} \quad \forall\,
\by,\bz \in B(\by_0,\varepsilon),
$$
which concludes the proof.
\end{proof}

\begin{remark}
In view of \eqref{first-lowerb} it follows that the quadratic growth
of the weight $(1+\bx_3^2)$ in \eqref{main-ub-eq} is sharp.    The lower
bound \eqref{first-lowerb} holds also for $t\leq 1$. However, for
small times the bounds \eqref{zhang} proved in \cite{zhang} are
sharper.
\end{remark}

It is also worth noticing that, because of a well--known
probabilistic interpretation of the Dirichlet heat kernel, the above
results can be reformulated in terms of the survival probability of
the Brownian bridge killed upon exiting $\Omega$. In fact we have
the following nonstandard asymptotic result.

\begin{corollary}
Let $\{X_s\}_{s\geq 0}$ be the Brownian loop process joining $x$ to
itself in time $t$ and let $P^{t,x,x}$ be the conditional Wiener
measure, normalized so that its total mass coincides with the free
heat kernel on ${\mathbb R}^3$. Then there exist strictly positive
constants $c_1,c_2$ such that, for all $\bx\in \Omega$ :
\begin{align*}
c_1\, (1+|\bx_3|^2)\, \rho^2(\bx)\,
& \, \leq \, \liminf_{t\to+\infty}
\Big[t^{\frac32}\, \mathrm{e}^{E_1t}\, P^{t,\bx,\bx}\left(X_s\in\Omega ~~\forall
s\in[0,t]\right)\big] \\
&\, \leq\, \limsup_{t\to+\infty} \Big[t^{\frac32}\, \mathrm{e}^{E_1t}\,
P^{t,\bx,\bx}\left(X_s\in\Omega ~~\forall s\in[0,t]\right)\Big]\\
&\leq\,  c_2\, (1+|\bx_3|^2)\, \rho^2(\bx).
\end{align*}
\end{corollary}

\smallskip

\subsection{Generalization}\label{sec-general}

As already mentioned in the Introduction, the method that we use to prove Theorem \ref{main-result} is applicable to a wide class of operators in $L^2(\Omega_0)$. To be more specific, let us consider nonnegative uniformly elliptic operators of the form
\begin{equation} \label{eq-L}
L\,f=  -\sum_{i,j=1}^3  \partial_{x_i} (a_{ij} (x) \partial_{x_j}\, f) + V(x) f,
\end{equation}
where $a_{ij}$ and $V$ are real valued functions, and $L$ is understood as the Friedrichs extension of the differential operator on the right hand side defined originally on $C_0^\infty(\Omega_0)$. We suppose that $a:=\big(a_{i j}\big)$  and $V$ are sufficiently smooth in $\overline{\Gw_0}$ and that $L = -\Delta -E_1$ for $|x_3|$ large enough. The arguments in the proof of Theorem  \ref{thm-lowerb-1} then immediately give a lower bound on $\mathrm{e}^{-t L}(x,x)$ given by the right hand side of \eqref{main-ub-eq} with $\bx$ replaced by $x$ and $\rho(\bx)$ replaced by dist$(x,\partial\Omega_0)$.

On the other hand, if we also suppose that $L$ is subcritical, then by Theorem \ref{greenf-equiv}  and \cite[Lemma 2.4]{p88} it follows that there exist smooth positive functions $u_j,\, j=0,1,2,$ such that $L\, u_j = 0$ in $\Gw_0$, and $u_j \asymp w_j$. Moreover, by the uniform ellipticity of $L$ we have
$$
( u_j\, \varphi, L (u_j\, \varphi) )_{L^2(\Omega_0)} \, =  \int_{\Omega_0} ( \nabla\varphi \cdot( a \nabla\varphi) )\, u_j^2\,  \mathrm{d}x \, \geq \, c \int_{\Omega_0} |\nabla\varphi|^2\, w_j^2\, \mathrm{d}x,
$$
for $j=0,1,2$. Hence a straightforward modification of Propositions \ref{hk-twisted} and \ref{non-twist-ub} gives

\begin{theorem} \label{thm-generalization}
Let $L$  be a uniformly elliptic operator of the form \eqref{eq-L}, and assume that
$a_{ij}$ and $V$ are H\"older continuous in $\overline{\Omega_0}$, and that $L=-\Delta-E_1$ in $\{(x',x_3)\in \Gw_0 \mid |x_3|>R \}$ for some $R>0$. Assume  further that $L$ is subcritical in $\Omega_0$. Then
\begin{equation}
\exp(-t L)(x,x)  \ \asymp  \ \frac{\mathrm{dist}(x,\partial\Omega_0)^2}{\sqrt t}\, \,
\min\Big\{\frac{1+x_3^2}{t}, \, \, 1\Big\}, \qquad \forall\ t \geq 1, \quad \forall\ x\in \Omega_0.
\end{equation}
\end{theorem}


\section{Integral estimates}
\label{sec-l2}

In this section we will prove certain integral estimates for the
semigroup $\mathrm{e}^{t\, (\Delta^D_\Omega+E_1)}$. We start with a
simple consequence of Theorem~\ref{main-ub}.

\begin{corollary} \label{interpolate}
There exists a constant $C$ such that for any $\mu,\, \nu\in [0,1]$
and any $t\geq 1$ it holds
\begin{equation} \label{interpolate-eq}
k(t,\bx,\by) \, \leq  \, C\, \rho(\bx)\,
\rho(\by)\, (1+\bx_3^2)^{\frac{\mu}{2}}\,
(1+\by_3^2)^{\frac{\nu}{2}} \, \, t^{-\frac{1+\mu+\nu}{2}} .
\end{equation}
\end{corollary}

\begin{proof}
We recall the following well--known inequality:
\begin{equation} \label{semigroup}
  k(t,\bx,\by)   \leq \sqrt{k(t,\bx,\bx)}\, \, \sqrt{k(t,\by,\by)}\, .
\end{equation}
For the convenience of the reader we briefly recall the proof of
this fact. By the semigroup property and the symmetry of the heat
kernel
\begin{equation} \label{k^2}
k(2t,\bx,\bx)=\int_\Omega k(t,\bx,\by)^2\,{\rm d}y.
\end{equation}
Hence, again by the symmetry, and the semigroup property, and in light of the Cauchy--Schwarz inequality, we
get
\begin{align}
k(2t,\bx,\by)&=\int_\Omega k(t,\bx,\bz)k(t,\by,\bz)\,{\rm d}z\le
\Big(\int_\Omega k(t,\bx,\bz)^2\,{\rm d}z\Big)^{1/2}
\Big(\int_\Omega k(t,\by,\bz)^2\,{\rm d}z\Big)^{1/2} \nonumber\\
&=    k(2t,\bx,\bx)^{1/2}k(2t,\by,\by)^{1/2}    \label{double-t}
\end{align}
as claimed. It now remains to apply Proposition \ref{non-twist-ub}
and Theorem~\ref{main-ub}.
\end{proof}

\noindent Now let us introduce the following family of weighted
$L^p$ spaces:
$$
L^p_\beta(\Omega) := \Big\{ f\, :  \|f\|_{L^p_\beta(\Omega)} <
\infty \Big\}, \quad \|f\|_{L^p_\beta(\Omega)} :=
\Big(\int_\Omega | f | ^p\, (1+\bx_3^2)^{\beta}\, \mathrm{d}\bx \Big)^{\frac
1p},\quad \beta\in\R.
$$
With this notation we have

\begin{proposition} \label{prop-l2}
For any $\kappa \in [0,2]$ and any $\beta >(1+\kappa)/2$ there
exists $C= C(\beta, \kappa)$ such that
\begin{equation} \label{l2-estim}
\| \mathrm{e}^{t\, (\Delta^D_\Omega+E_1)}\|_{L^2_\beta(\Omega)\to
L^2(\Omega)} \, \leq \, C\, (1+t)^{-\frac{1+\kappa}{4}}\qquad
\forall\, t\geq 0.
\end{equation}
\end{proposition}

\begin{proof}
Let $f\in L^2(\Omega)$. In view of \eqref{k^2}, Cauchy-Schwarz inequality and \eqref{interpolate-eq} applied with $\mu=\nu =\kappa/2$ we get
\begin{align}
\| \mathrm{e}^{t\, (\Delta^D_\Omega+E_1)}\, f \|^2_{L^2(\Omega)} &
 \leq \ \|f\|^2_{L^2_\beta(\Omega)}\!  \int_{\Omega\times\Omega}\! k(t, \bx, \by)^2 \, (1+\by_3^2)^{-\beta}
\, \mathrm{d}\bx\, \mathrm{d}\by \nonumber \\
& = \|f\|^2_{L^2_\beta(\Omega)}\!  \int_{\Omega}\! k(2t, \by, \by) \, (1+\by_3^2)^{-\beta}
\, \mathrm{d}\by\, \leq \tilde{C}\,
t^{-\frac{1+\kappa}{2}}\, \, \|f\|^2_{L^2_\beta(\Omega)} \nonumber
\end{align}
for all $t\geq 1$. This shows that
\begin{equation} \label{norm}
\| \mathrm{e}^{t\, (\Delta^D_\Omega+E_1)}\|_{L^2_{\beta}(\Omega)\to
L^2(\Omega)}  \, \leq \, C\, t^{-\frac{1+\kappa}{4}} \qquad \forall \ t \geq 1.
\end{equation}
Equation \eqref{l2-estim} then follows from
\eqref{norm} and from the fact that  $\mathrm{e}^{t\,
(\Delta^D_\Omega+E_1)}$ is, for all $t\geq 0$, a contraction from
$L^2(\Omega)$ to $L^2(\Omega)$.
\end{proof}

\begin{remark}
Proposition~\ref{prop-l2} with $\kappa=2$ extends inequality \eqref{kz} to the case $a=3/4$ and
at the same time allows for a much slower growth of the integral weight than the one used in \eqref{kz},
namely $(1+\bx_3^2)^{\beta}$ with any $\beta>3/2$.
On the other hand, the corresponding estimate in \cite{kz}
was obtained under weaker regularity assumptions on $\theta$.
\end{remark}

\noindent The following estimate is a version of Proposition \ref{prop-l2} in suitable $L^1$ and $L^\infty$ spaces. In order to state it we  introduce   for $\beta\geq 0$ the spaces
$$
 L^\infty_{-\beta}(\Omega) = \big\{ f\, :  \|f\|_{L^\infty_{-\beta}(\Omega)} := \| (1+\bx_3^2)^{-\beta}\, f\|_{L^\infty(\Omega)}  < \infty \big\} .
$$

\noindent We then have

\begin{theorem} \label{thm-lp}
For any  $\beta \in [0,1/2]$  we have
\begin{equation} \label{lp-estim}
\| \mathrm{e}^{t\, (\Delta^D_\Omega+E_1)}\|_{L^2(\Omega)\to L^\infty_{-\beta}(\Omega)} = \| \mathrm{e}^{t\, (\Delta^D_\Omega+E_1)}\|_{L^1_\beta(\Omega)\to
L^2(\Omega)}  \,\asymp \, t^{-\frac 14-\beta} \qquad \forall \ t \geq 1.
\end{equation}
\end{theorem}

\begin{proof}
The equality in \eqref{lp-estim} follows by duality using the scalar product $(u,v)= \int_{\Omega} \bar u v\, \mathrm{d}\bx$ in $L^2(\Omega)$.
Let $f\in L^2(\Omega)$. By \eqref{k^2}, Cauchy-Schwarz inequality and estimate
\eqref{interpolate-eq} with $\mu=\nu =2\beta$ we obtain
\begin{align*}
\| \mathrm{e}^{t\, (\Delta^D_\Omega+E_1)}\, f \|_{L^\infty_{-\beta}(\Omega)} & \leq    \|f\|_{L^2(\Omega)}\, \|(1+\bx_3^2)^{-\beta}\, \sqrt{k(2t,\bx,\bx)} \, \|_{L^\infty(\Omega)}    \\
& \leq \, C_\beta\ t^{-\frac 14-\beta} \ \|f\|_{L^2(\Omega)}.
\end{align*}
This proves the upper bound in \eqref{lp-estim}. To prove
the lower bound let us consider  a generalized function $f_t$ given by a Dirac delta distribution placed in a point $\bz(t)\in\Omega$ such that
$1+\bz_3^2(t) = 2 t$ and $\rho(\bz (t))>\varepsilon>0$ for all t. From \eqref{main-ub-eq} and  \eqref{k^2} it then follows that
$$
 \frac{\| \mathrm{e}^{t\, (\Delta^D_\Omega+E_1)}\, f_t \|_{L^2(\Omega)} }{\|f_t\|_{L^1_\beta(\Omega)}} \, =\,  \frac{\sqrt{k(2t, \bz(t), \bz(t))}}{(1+\bz_3^2(t))^\beta}
 \, \geq \, C\, t^{-\frac 14-\beta}.
$$
\end{proof}

\begin{remark}
In the absence of twisting we have
\[\begin{split}
 \| \mathrm{e}^{t\, (\Delta^D_{\Omega_0}+E_1)}\|_{L^2(\Omega_0)\to L^\infty_{-\beta}(\Omega_0)} =  \| \mathrm{e}^{t\, (\Delta^D_{   \Omega_0   }+E_1)} \|_{L^1_\beta(\Omega_0)\to
L^2(\Omega_0)} \ \asymp \ t^{-\frac 14}\\
\forall \ t \geq 1, \; \forall\ \beta \geq 0,
\end{split}
\]
which can be easily derived from the explicit expression for the integral kernel of  $\mathrm{e}^{t\, (\Delta^D_{   \Omega_0   }+E_1)}$ in $\Omega_0$. Notice also that proceeding as in the proof of Theorem \ref{thm-lp} but choosing $\bz(t)=\mathrm{constant}$ shows that no matter how large $\beta$ is, the left hand side of \eqref{lp-estim} will never decay faster than $t^{-\frac 34}$.
\end{remark}

\section{Spectral estimates and Sobolev inequality}
\label{sec-clr-sobol}

Let $V: \Omega\to \R$ be a real valued measurable function and
consider the Schr\"odinger operator
$$
-\Delta_\Omega^D - V \qquad \text{in\, } L^2(\Omega)
$$
associated with the quadratic form
\begin{equation}
\int_\Omega\, \left(|\nabla u|^2 -V |u|^2\right)\, \mathrm{d}\bx, \quad u\in
H^1_0(\Omega).
\end{equation}
Let us denote by $N(-\Delta_\Omega^D - V, s)$ the number of discrete
eigenvalues of $-\Delta_\Omega^D - V$ less than $s$ (counted with
multiplicity). If $V=0$, then of course $N(-\Delta_\Omega^D,
E_1)=0$. In the problems concerning spectral estimates one usually tries to control
$N(-\Delta_\Omega^D - V, E_1)$ in terms of $V$.

Without loss of generality we may assume that $V\geq 0$ (otherwise
we replace $V$ by $V_+$). By the Lieb's inequality, see \cite{lieb}, \cite{fls}, \cite{roso2}, we have
\begin{align}
N(-\Delta_\Omega^D - V, E_1) & =N(-\Delta_\Omega^D -E_1-V, 0) \, \nonumber \\
& \leq \, M_{b}\, \int_\Omega\, \int_{0}^\infty  k(t,\bx,\bx)\,
t^{-1}\, (t\, V(\bx)-b)_+\, \mathrm{d}t\, \mathrm{d}\bx, \label{lieb}
\end{align}
where $b>0$ is arbitrary and
\begin{equation*}
M_b= \, \big( \mathrm{e}^{-b}-b\, \int_b^\infty\, s^{-1}\, \mathrm{e}^{-s}\,
\mathrm{d}s\big)^{-1}.
\end{equation*}

\noindent From inequality \eqref{zhang} and Theorem~\ref{main-ub}
follows that there exists a constant $C$ such that for all $t>0$
and all $\bx\in\Omega$ it holds
 \begin{equation} \label{global}
k(t,\bx,\bx) \, \leq  \, C\, (1+\bx_3^2)\, \, t^{-\frac 32} .
\end{equation}
A direct application of \eqref{global} and \eqref{lieb} then gives

\begin{theorem} \label{CLR}
There exists a positive constant $L$ such that
\begin{equation} \label{CLR-eq}
N(-\Delta_\Omega^D - V, E_1) \, \leq \, L \int_\Omega\, V^{\frac
32}(\bx)\, (1+\bx_3^2)\, \mathrm{d}\bx
\end{equation}
holds for all $0\leq V\in L^{3/2}(\Omega, (1+\bx_3^2))$.
\end{theorem}

\begin{remark}
   Due to the criticality of  $-\Delta_{\Omega_0}^D - E_1$,    inequality \eqref{CLR-eq} fails in the absence of twisting since
$$
N(-\Delta_{\Omega_0}^D - \alpha\, V, E_1) \geq 1 \qquad \forall\,
\alpha >0
$$
provided $V\gneqq 0$ satisfies    the    assumptions of Theorem \ref{CLR}, see
  \cite{EW02,pt,roso}.    Note also that the bound \eqref{CLR-eq} has the
right semiclassical behavior since it is well-known (see e.g.
\cite{roso}) that
$$
N(-\Delta_\Omega^D - \alpha\, V, \, E_1) \ \asymp \
\alpha^{3/2}  \qquad \alpha\to\infty.
$$
\end{remark}

\begin{remark}
It also easy to see that the weight $(1+\bx_3^2)$ in \eqref{CLR-eq}
cannot be improved in the power-like scale. For if $V(\bx) \asymp
|\bx_3|^{-2+\varepsilon}$ as $|\bx| \to \infty$ with some $\varepsilon>0$, then a
standard test function argument, cf.\! \cite[Theorem~13.6]{rs4},  shows that
$$
N(-\Delta_{\Omega_0}^D - \alpha\, V, E_1) =\infty \qquad \forall\ \alpha >0.
$$
\end{remark}

\noindent Estimate \eqref{CLR-eq} in combination with Hardy inequality \eqref{hardy-operator} yield the following family of
weighted Sobolev inequalities,  which have no analogue in the straight tube $\Omega_0$.

\begin{theorem} \label{thm-sobolev}
For any $p\in [2,6]$ there exists a constant $C_p>0$ such that
\begin{equation} \label{sobolev-eq}
\int_{\Omega} \left(|\nabla u|^2-E_1\, |u|^2\right)\, \mathrm{d}\bx \ \geq \ C_p\
\Big( \int_{\Omega}  |u|^p \, (1+\bx_3^2)^{-\frac{p+2}{4}}\, \mathrm{d}\bx\Big)^{2/p}
\end{equation}
holds for all $u\in H^1_0(\Omega)$.
\end{theorem}

\begin{proof}
First we mimic the argument used in \cite{fls} and note that by
\eqref{CLR-eq}
\begin{align}
& L \int_\Omega\, V^{3/2}(\bx)\, (1+\bx_3^2)\, \mathrm{d}\bx  <1 \, \,
\label{cond}  \\
& \qquad \qquad\qquad\qquad \Big\Downarrow  \nonumber \\
&\int_\Omega\, \left(|\nabla u|^2 -E_1\, |u|^2-V |u|^2\right) \mathrm{d}\bx \geq 0 \qquad  \forall\,
u\in H^1_0(\Omega). \label{impl}
\end{align}
Let $u\in H^1_0(\Omega)$. Choosing
$$
V(\bx)= \eta\, |u|^4\, (1+\bx_3^2)^{-2} \, \Big(L \int_\Omega\, |u|^6\,
(1+\bx_3^2)^{-2}\, \mathrm{d}\bx\Big)^{-\frac 23}
$$
with $\eta<1$ we see that \eqref{cond} is satisfied and \eqref{impl}
gives
$$
\int_{\Omega} \left(|\nabla u|^2-E_1\, |u|^2\right)\, \mathrm{d}\bx\,  \geq\,  C\,
\Big( \int_{\Omega}  |u|^6 \, (1+\bx_3^2)^{-2}\, \mathrm{d}\bx\Big)^{1/3}
$$
for some $C>0$. This together with H\"older inequality and \eqref{hardy-operator} implies
\begin{align*}
\int_{\Omega}  |u|^p \, (1+\bx_3^2)^{-\frac{p+2}{4}}\, \mathrm{d}\bx\, & \leq\,  \Big( \int_{\Omega}  |u|^6 \, (1+\bx_3^2)^{-2}\, \mathrm{d}\bx\Big)^{\frac{p-2}{4}}\, \Big( \int_{\Omega}  |u|^2 \, (1+\bx_3^2)^{-1}\, \mathrm{d}\bx\Big)^{\frac{6-p} {4}} \\
& \leq \, c\  \Big( \int_{\Omega} \left(|\nabla u|^2-E_1\, |u|^2\right)\, \mathrm{d}\bx \Big )^{\frac p2},
\end{align*}
as claimed.
\end{proof}

\noindent Let us define the sequence of functions
$g_n:\R\to\R$ by
$$
g_n(s) := \left \{
\begin {array}{lr}
1-\frac 1n\,  (s+R),  &    \quad -R-n \leq s < -R,  \\[2mm]
1, & \quad -R \leq s \leq R, \\[2mm]
1 -\frac 1n\,  (s-R), &  \quad  R < s < R+n,
\end {array}
\right.
$$
and $g_n=0$ otherwise, recalling that $\mathrm{supp}\, \dot\theta\subset
(-R,R)$. We make  the  following observations.

\begin{remark}
   Due to the criticality of $-\Delta^D_{\Omega_0}-E_1$,  inequality \eqref{sobolev-eq} fails if $\dot\theta\equiv 0$ \cite{pt09}.    Indeed, the
choice $u_n(x)= \psi_1(x_1,x_2)\, g_n(x_3)$
gives $u_n \in H^1_0(\Omega_0)$ and
$$
\int_{\Omega_0} \left(|\nabla u_n|^2-E_1\,   |u_n|^2  \right)\, {\rm d}x =
\int_\R |g'_n(x_3)|^2\, {\rm d}x_3 = \mathcal{O}(n^{-1})\qquad n\to\infty,
$$
while
$$
\int_{\Omega_0}  |u_n|^p \, (1+x_3^2)^{-\frac{p+2}{4}}\,  {\rm d}x \to \int_\R
(1+x_3^2)^{-\frac{p+2}{4}}\, {\rm d}x_3 \qquad n\to\infty
$$
by the monotone convergence theorem. This will be in contradiction with
\eqref{sobolev-eq} if we  replace  $\Omega$ with $\Omega_0$.
\end{remark}

\begin{remark} \label{optim}
  The decay rate    of the weight in
the integral on the right hand side of \eqref{sobolev-eq} cannot be
improved in the power-like scale. In other words, the inequality
\begin{equation}\label{sobol-modif}
\int_{\Omega} \left(|\nabla u|^2-E_1 |u|^2\right)\, \mathrm{d}\bx\  \geq \ C\, \Big(
\int_{\Omega}  |u|^p \, (1+\bx_3^2)^{-\gamma}\, \mathrm{d}\bx\Big)^{\frac 2p}
\qquad \forall\, u\in H^1_0(\Omega)
\end{equation}
fails whenever $\gamma <(p+2)/4$. To see this we use the sequence of test
functions $u_n(\bx)= v_0(r_\theta^{-1} \bx)\, g_n(\bx_3)$. Then
$u_n\in H^1_0(\Omega)$ and using \eqref{gr-eq} we get
$$
\int_{\Omega} \left(|\nabla u_n|^2-E_1\, |u_n|^2\right)\, \mathrm{d}\bx \, \leq\,
C \int_\R (1+\bx_3^2)\ |g'_n(\bx_3)|^2\, \mathrm{d}\bx_3 = \mathcal{O}(n),
$$
and
$$
\Big( \int_{\Omega}  |u_n|^p \, (1+\bx_3^2)^{-\gamma}\,
\mathrm{d}\bx\Big)^{\frac 2p} \
\geq \ C\, n^{\frac{2(p+1-2\gamma)}{p}} \, +
o\big(n^{\frac{2(p+1-2\gamma)}{p}}\big)
$$
as $n\to\infty$. Hence from  \eqref{sobol-modif} it follows that
$\gamma  \geq (p+2)/4$.
\end{remark}



\medskip

\begin{center}{\bf Appendix} \end{center}

\begin{appendix}
\section{One-dimensional Schr\"odinger operators}
\label{sect-1dim}

\noindent In this section we prove an auxiliary result concerning the semigroup generated by the nonnegative operator
\begin{equation*}
P = -\frac{d^2}{dr^2}\, + \dot\theta^2(r) \quad \text{in\, } L^2(\R).
\end{equation*}

\begin{proposition}[cf. {\cite[Theorem~4.2]{m84}}] \label{prop-app}
There exists a constant $c$ such that
\begin{equation} \label{1dim-appendix}
q(t,r,r):= \mathrm{e}^{-t P}(r,r)  \, \leq \, c\ \min\Big\{\frac{g_0^2(r)}{t^{3/2}}\, , \, \frac{1}{\sqrt{t}}\Big\} \qquad \forall\ t >0,
\end{equation}
where $g_0$ is given by \eqref{cauchy}.
\end{proposition}

\begin{proof}
One estimate follows immediately by the Trotter product formula:
$$
\mathrm{e}^{-t P}(r,r)  \, \leq \, \exp \Big (\displaystyle{\,t \frac{d^2}{dr^2}}\Big) (r,r) = (4\pi t)^{-1/2} .
$$
To prove the remaining part of \eqref{1dim-appendix}       we note that
since $P g_0 = 0$, the operator $\widetilde P :=  g_0^{-1}\, P\, g_0$ in $L^2(\R, g_0^2(r) \mathrm{d}r)$ is associated with the quadratic form
$$
\int_\R |f'(r)|^2\, g_0^2(r)\, \mathrm{d}r, \qquad f\in H^1(\R, g_0^2(r)\, \mathrm{d}r),
$$
and the corresponding semigroup $\mathrm{e}^{-t \widetilde P}$ satisfies
$$
\mathrm{e}^{-t P}(r,r') = g_0(r)\, g_0(r')\ \mathrm{e}^{-t \widetilde P}(r,r').
$$
Hence it suffices to show that
\begin{equation} \label{aux-app}
\sup_{r>0}\, \mathrm{e}^{-t \widetilde P}(r,r)\, \leq \, c\,t^{-3/2} \qquad \forall \ t >0.
\end{equation}
By the well-known Theorem of Varopoulos, see e.g. \cite[Theorem~2.4.2]{da}, estimate \eqref{aux-app} will follow from the Sobolev inequality
\begin{equation} \label{sob-app}
\int_\R |f'(r)|^2\, g_0^2(r)\, \mathrm{d}r \, \geq \, c_s\, \Big(\int_\R |f(r)|^6\, g_0^2(r)\, \mathrm{d}r\Big)^{1/3} \qquad \forall\ f\in H^1(\R, g_0^2(r)\, \mathrm{d}r).
\end{equation}
To prove \eqref{sob-app} we consider a function $u\in H^1(\R_+, (1+r)^2\, \mathrm{d}r)$. Integration by parts yields the identity
\begin{equation} \label{hardy-app}
\int_0^\infty \Big(u' +\frac{u}{2(1+r)}\Big)^2\, (1+r)^2\, \mathrm{d}r = \int_0^\infty |u'|^2 (1+r)^2\, \mathrm{d}r -\frac{u^2(0)}{2} - \frac 14\, \int_0^\infty |u|^2\, \mathrm{d}r.
\end{equation}
Moreover,
$$
|u(r)|^2 = -2    \int_r^\infty    u'(s)\, u(s)\, \mathrm{d}s \, \leq \int_0^\infty |u'(r)|^2\, \mathrm{d}r+\int_0^\infty |u(r)|^2\, \mathrm{d}r.
$$
This in combination with \eqref{hardy-app} and the H\"older inequality gives
\begin{align} \label{sob-1}
\int_0^\infty |u'|^2 (1+r)^2\, \mathrm{d}r \, \geq \,    \frac 18  \int_0^\infty |u'|^2\, \mathrm{d}r + \frac 12 \int_0^\infty |u'|^2 (1+r)^2\, \mathrm{d}r    \nonumber \\[3mm]
\geq \frac 18  \int_0^\infty |u'|^2\, \mathrm{d}r + \frac 18 \int_0^\infty |u|^2\, \mathrm{d}r \geq \, \frac 18 \|u\|_\infty^2 \geq\, \frac 18\, \Big(\int_0^1 |u|^6\, \mathrm{d}r\Big)^{1/3}.
\end{align}
On the other hand, restriction of the standard Sobolev inequality in $\R^3$ with the critical exponent $q=6$ onto the subspace of radial functions gives
\begin{equation} \label{sob-2}
\int_0^\infty |u'|^2 \, r^2\, \mathrm{d}r \, \geq \, \tilde c\ \Big(\int_0^\infty |u|^6\, r^2\, \mathrm{d}r\Big)^{1/3}.
\end{equation}
Hence inequality \eqref{sob-app} follows from \eqref{sob-1}, \eqref{sob-2} and from the fact that $g_0^2(r) \asymp (1+|r|)^2 \asymp 1+r^2 $ on $\R$.
\end{proof}


\section{Properties of the functions $\xi_\lambda$ and $\vartheta_\lambda$}
\label{sect-xi}

\begin{lemma} \label{xi-lem}
Let $\xi,\vartheta$ be the functions defined by \eqref{xi}, \eqref{theta}.  For any
$\kappa>0$ there exists a constant $c_\kappa$ such that for all
$r>0$ and all  $\gl\geq 1 $   it holds
\begin{equation} \label{xi-estim}
\xi_\lambda(\kappa\, r) \leq C_\kappa\, \xi_\lambda(r), \qquad \vartheta_\lambda(\kappa\, r) \leq C_\kappa\, \vartheta_\lambda(r).
\end{equation}
\end{lemma}

\begin{proof}
Since $\xi_\lambda$ is increasing, we may assume that $\kappa>1$. A
straightforward calculation gives
$$
\xi_\lambda(r)  = \left\{
\begin{array}{l@{\qquad }cl}
\frac 32\, \lambda^{-\frac 23}\, r^{\frac 53} , & 0 < r \leq \lambda
\, ,
\\ & \\
-\lambda \, \chi'(\chi^{-1}(r/\lambda )) , &  \lambda < r < 2^{\frac 52}\, \lambda, \\
&\\
\frac 52\, \lambda^{-\frac 25}\, r^{\frac 75} , &  2^{\frac 52}\,
\lambda \leq r < \infty,
\end{array} \right.
$$
and
$$
 \
\vartheta_\lambda(r)  = \left\{
\begin{array}{l@{\qquad }cl}
\frac 12\, \lambda^{-2}\, r^3 , & 0 < r \leq \lambda
\, ,
\\ & \\
-\lambda \, \tilde\chi'(\tilde\chi^{-1}(r/\lambda )) , &  \lambda < r < 2^{\frac 52}\, \lambda, \\
&\\
\frac 52\, \lambda^{-\frac 25}\, r^{\frac 75} , &  2^{\frac 52}\,
\lambda \leq r < \infty.
\end{array} \right.
$$
\noindent It can be now directly verified that $\xi_\lambda$ and $\vartheta_\lambda$ satisfy
\eqref{xi-estim}.
\end{proof}

\section{Remark on Davies' conjecture}

Large time behaviors of the heat kernel and in particular sharp two-sided heat kernel estimates are closely related to the following conjecture.
\begin{conjecture}[{Davies' Conjecture \cite{da97}}]\label{conjD}
Let $M$ be a noncompact Riemannian manifold, and consider a linear, time independent, second-order parabolic operator of the form $$u_t+P(x, \partial_x)u$$  which is
defined on $M$. Assume that $E_1=E_1(P,M)$,  the generalized principal eigenvalue of the elliptic operator $P$ in $M$, is nonnegative. Let $k_P^M(x,y,t)$ be the corresponding positive minimal heat kernel. Fix reference
points $x_0, y_0\!\in\! M$.

Then
\begin{equation}\label{eqconjD}
\lim_{t\to\infty}\frac{k_P^M(x,y,t)}{k_P^M(x_0,y_0,t)}=a(x,y)
\end{equation}
exists and is positive for all $x,y\in M$ (see also \cite{p06,fkp} and the references therein).

\end{conjecture}
Recall that Davies' conjecture holds if $P-E_1$ is critical in $M$ and the product of the corresponding ground states is in $L^1(M)$. Moreover, it holds true in the {\em symmetric} case
if the cone of all positive solutions of the equation $(P-E_1)u=0$ that vanish on $\partial M$ is one-dimensional. Hence, it holds true for a critical symmetric operator. In particular,
\begin{equation*}\label{eqconjD4}
\lim_{t\to\infty}\frac{\mathrm{e}^{t \Delta^D_{\Omega_0}}(x,y)}{\mathrm{e}^{t \Delta^D_{\Omega_0}}(0,0)}=C\psi_1(x_1,x_2)\psi_1(y_1,y_2)
\end{equation*}

In the following remark we consider Davies' conjecture  in the present situation.
\begin{remark}\label{rem-davies}
It follows from \cite[Theorem~4.2]{m84} that Davies' conjecture holds true for Schr\"odinger operators on $\R$ provided the potential satisfies Murata's assumptions in \cite{m84}.   Clearly,
\begin{equation}\label{limit2}
\lim_{t\to\infty} \Big(\sum_{j=1}^\infty\, \mathrm{e}^{t(E_1-E_j)}\, \psi_j(x_1,x_2)\psi_j(y_1,y_2)\Big)=\psi_1(x_1,x_2)\psi_1(y_1,y_2).
   \end{equation}
Using the heat kernel  decomposition \eqref{hk-decomp}, \eqref{limit2} and \cite[Theorem~4.2]{m84}, it follows that
Davies' conjecture holds true for the operator $A$ on $\Gw_0$, where $A$ is the subcritical operator defined by \eqref{model}. The validity of Davies's conjecture for operators $L$ satisfying the assumptions of Theorem~\ref{thm-generalization} and for the Laplacian on a twisted tube remains open.
\end{remark}


%
%
\end{appendix}

\section*{acknowledgements}
The authors wish to thank the anonymous referee for his/her valuable and
detailed remarks and suggestions.



\end{document}